\documentclass[11pt]{amsart}
\usepackage[T1]{fontenc}
\usepackage[american]{babel}
\usepackage{csquotes}
\usepackage{verbatim}
\usepackage{amsmath}
\usepackage{amsfonts}
\usepackage{amssymb}
\usepackage{gensymb}
\usepackage{cmll} 
\usepackage{accents}
\usepackage{amsthm}
\usepackage[hyphens]{url}
\usepackage{hyperref}
\usepackage{breakurl}
\usepackage[style=numeric]{biblatex}
\bibliography{bib}

\theoremstyle{plain} 
\newtheorem{theorem}{Theorem}

\newtheorem{prop}[theorem]{Proposition}
\newtheorem{corollary}[theorem]{Corollary}

\theoremstyle{definition}

\newtheorem{example}[theorem]{Example}

\newcommand{\LoP}{{\mathrm{LP}}}
\newcommand{\JL}{{\mathrm{ JL}}}
\newcommand{\DL}{{\mathrm{ DL}}}
\newcommand{\DLP}{\DL\oplus_{\JL}\LoP}
\newcommand{\DLJP}{\DL\oplus_{\JL}\JL\oplus_{\JL}\LoP}
\newcommand{\OK}[1]{\mathrm{OK}\left({#1}\right)}

\newcommand{\PROVES}{\vdash}
\newcommand{\MODELS}{\vDash}
\newcommand{\AND}{\wedge}
\newcommand{\OR}{\vee}
\newcommand{\NOT}{\lnot}
\newcommand{\IMP}{\rightarrow}    

\newcommand{\FALSE}{\bot}

\newcommand{\BANG}{\oc}  
\newcommand{\WITH}{\with}  

\newcommand{\JUST}[2]{{#1}:{#2}}  

\newcommand{\DCOLON}{\mathrel{::}}

\title{Denial Logic}
\author[F. Lengyel]{$^\dagger$Florian Lengyel}
\address{$^\dagger$CUNY Environmental CrossRoads Initiative\\
The City College of New York, CUNY\\
160 Convent Avenue\\ 
New York 10031}
\address{$^\dagger$CUNY Graduate Center\\
365 Fifth Avenue\\
New York 10016}

\author[Ben St-Pierre]{$^\ddagger$Benoit St-Pierre}
\address{$^\ddagger$Soeurs de la Charit\'e de Montr\'eal}
\email{$^\dagger$flengyel@ccny.cuny.edu, flengyel@gc.cuny.edu}
\email{$^\ddagger$ben@oueb.ca}
\date{March 16, 2012}

\begin{document}

\begin{abstract}
Denial Logic is the logic of an agent whose justified beliefs are
false, who cannot avow his own propositional attitudes or believe
tautologies, but who can believe contradictions.  
Denial Logic $\DL$ is defined as  justification logic $\JL$ 
together with the Denial axiom $\JUST t E\IMP \NOT E$ and the 
Evidence Pairing axiom $\JUST s D\AND\JUST t E\IMP\JUST{[s\WITH t]}{D\AND E}$. 
Using Artemov's  natural  $\JL$ semantics,  in which justifications are 
interpreted  as sets of formulas, we provide an inductive construction of 
models of $\DL$, and show that $\DL$ is sound and complete.   
Some notions developed for $\JL$, such as constant specifications 
and internalization, are inconsistent with  $\DL$.  
In contrast, we define  negative  constant specifications, 
which can be used to model agents in $\DL$ with justified false beliefs. 
Denial logic can therefore be relevant to philosophical skepticism. 
We define coherent negative constant specifications for $\DL$ to
model a Putnamian brain in a vat with the justified false belief that 
it is not a brain in a vat,  and prove a ``Blue Pill'' theorem, which 
produces a  model of $\JL$ in which ``I am a brain in a vat'' is false.   
We extend $\DL$ to the  multi-modal logic $\DLP$ to model  envatted brains 
who can justify and check  tautologies and avow their own propositional 
attitudes.  Denial Logic was inspired by online debates over anthropogenic 
global warming.    
\medskip

\end{abstract}
\maketitle

\section*{Introduction}
This paper is a contribution to the study of logics of skepticism. Our setting is 
Justification Logic $\JL$, the minimal logic for a family of logics that includes  the 
Logic of Proofs $\LoP$, a Hilbert-style logical system  extending classical  and 
intuitionistic propositional logic, with additional  propositional types of the form 
$\left(\JUST{t}{A}\right)$, read as ``term $t$ is justification  for $A$'' \cite{MR1836474,MR1996579}. 
 $\LoP$ and the related broader class of justification logics are, in a precise sense, refinements 
of  modal logics, including  \textrm{K}, \textrm{K4}, \textrm{K45}, \textrm{KD45}, 
\textrm{T}, \textrm{S4} and \textrm{S5}   \cite{MR1836474,RSLDec2008,MR1996579,Art2008,Art2012}. 
Semantics for these systems include Kripke--Fitting models, Mkrtychev models, and arithmetical 
provability semantics \cite{MR2102853, MR1611428, MR1836474}.  Recently, Artemov provided 
Justification Logic  with a natural semantics in which justifications are interpreted as sets 
of formulas \cite{Art2012}. In addition to its applications 
in epistemic logic, modal logic and proof theory, $\LoP$ has been 
generalized to interactive multi-agent computation \cite{Kramer2012}.

Studies  of justification logics have tended to focus on justified true belief, provability and 
logical omniscience \cite{MF2009,MR1836474,SARK2009}. Some notions defined for justification 
logics, such as axiomatically appropriate constant specifications, factivity and logical 
introspection, reflect this focus \cite{SATR2010,Art2008,RSLDec2008}. Here we study a  logic 
of justified false belief, Denial Logic $\DL$, defined as $\JL$ together  with the Denial axiom  
$\JUST t E\IMP \NOT E$ and the Evidence Pairing axiom 
$\JUST s D\AND\JUST t E\IMP\JUST{[s\WITH t]}{D\AND E}$.  $\DL$  is 
the logic of an agent whose justified beliefs are
false, who cannot avow his own propositional attitudes or believe
tautologies, but who can believe contradictions.

This paper is organized as follows. The first section defines the syntax of 
Denial Logic and shows that $\DL$ cannot justify any of its axioms. We 
observe that  notion of constant specification 
for $\JL$ can lead to inconsistency in $\DL$. Accordingly, we define the notion of negative constant 
specification. 

The second section extends Artemov's natural semantics for $\JL$ to $\DL$. We define the notion of 
model for $\DL$, prove completeness, and  give an inductive construction 
of models of $\DL$.  We apply  the inductive 
construction of models of $\DL$ to produce models satisfying  negative constant specifications and 
to prove that  the Evidence Pairing axiom is independent of the other axioms of $\DL$.

The third section applies  $\DL$ to Putnam's brains in 
vats where we prove a ``Blue Pill'' theorem \cite{RBlackford2004}. 
Here we use justification logic 
to define a formal notion of coherence suggested by philosophical coherence 
theories of truth \cite{Davidson86}.  

In the fourth section, we note that  a $\DL$ agent cannot believe tautologies 
of propositional logic or avow any of his propositional attitudes.  To handle 
agents that can,  we extend $\DL$ to $\DLP$, the algebraic fibring of $\DL$ 
with the Logic of Proofs $\LoP$ constrained by justification logic $\JL$.

In the fifth section we apply $\DLP$ to an agent who denies that climate 
models indicate anything true, but who allows that $CO_2$ is a greenhouse gas. 
The Denial axiom of the $\DL$ fragment of $\DLP$ is used to model the agent's 
assertion that every indicator produced by  a climate model is wrong. 
This agent cannot provably justify any any scientific conclusion that might 
follow from the concession that $CO_2$ is a greenhouse gas.

In the last section we observe that the Blue Pill theorem extends to $\DLP$.

\subsection{Related work} 
Epistemic logics have been used and developed for the analysis of philosophical skepticism.
Steiner treats Cartesian skepticism in a normal modal logic and shows that a strong skeptical
argument remains if the KK axiom is dropped \cite{MS1979}. Schotch and Jennings propose an alternative
to possible-world semantics and define Basic Epistemic Logic, a non-normal modal logic developed 
to address problems of logical omniscience and skepticism  \cite{PJ1981}. Schotch   defines the  
proto-epistemic logics, in which the knowledge modality need not distribute over  implication \cite{PKS2000}.  
The application here of justification logic to philosophical skepticism  appears to be new.

\section{Denial Logic}
The language  of $\DL$ is the same as that of justification logic $\JL$, with 
the addition of a binary operation $\WITH$ on justification terms \cite{Art2008}. Symbols of $\DL$ include those occurring in justification terms as follows.
  \begin{itemize}
  \item[] {\it justification constants} $c, c_1,c_2,\ldots$
  \item[] {\it justification variables} $u,v,w,x,y,z, x_1,x_2,\ldots$
  \item[] {\it binary operations} $\cdot$, $+$, $\WITH$
    \item[] {\it punctuation} $[$, $]$ 
  \end{itemize}
The set of justification constants, which may be finite or infinite, is specified through what we will call a negative constant specification (see subsection \ref{SubNCS}).

A {\it justification term} $t$ is an expression of the form
\[
t \DCOLON= c_i\, |\, x_j\,
|\, \left[t+t\right]\,
|\, \left[t\cdot t\right]
|\, \left[t\WITH t\right]
\]
where $c_i$ is a justification constant and $x_j$ is a justification variable. 

Symbols occurring in formulas include justification terms and the following. 
  \begin{itemize}
  \item[] {\it propositional variables} $A,B,C,\ldots,X,Y,Z,A_1,A_2,\ldots$
  \item[] {\it propositional constant}  $\FALSE$
    \item[] {\it unary connective} $\NOT$
  \item[]{\it binary connectives} $\AND$,$\OR$,$\IMP$
  \item[] {\it punctuation} $($,$)$,$:$
  \end{itemize}

A {\it formula} $A$ of a $\DL$  is an expression of the form
\[
A \DCOLON= \FALSE\,  |\, A_i \, |\, \left(A\AND A\right) \,| \left(A\OR A\right) \,|\, \left(A\IMP A\right) \,|\,  \JUST t A
\]
where $t$ is a justification term and where $A_i$ is a propositional variable.  It should be noted 
that in the formula $\JUST t P$, the justification term $t$ justifies specific propositional syntax: 
$\JUST t P$ may hold but  $\JUST t {P \AND P}$  may fail in some model of $\JL$.

The axioms of $\DL$ include axioms of classical logic, the Application and 
Sum axioms of $\JL$, the Denial axiom and the Evidence Pairing axiom:
\begin{enumerate}
\item Application. $\JUST s {\left(P\IMP Q\right)} \IMP (\JUST t P \IMP \JUST {[s\cdot t]} Q)$ 
	\label{itm:AxApp}
\item Sum. $\JUST s P \IMP \JUST {[s+t]} P,\quad
                        \JUST t P \IMP \JUST {[s + t]} P$
	\label{itm:AxSum}
\item Denial. $\JUST t P\IMP \NOT P$ \label{itm:AxDenial}
\item Evidence Pairing. $\left(\JUST s P\AND\JUST t Q\right)
\IMP \JUST {[s\WITH t]}{\left(P\AND Q\right)}$\label{itm:AxPairing}
\end{enumerate}
in which $s$ and $t$ are justification terms. The Evidence Pairing Axiom is redundant in 
justification logics with axiomatically complete constant specifications but not in $\DL$ 
(cf. Theorem \ref{ThmInternal} and the remarks following for definitions). 

The rule of inference of $\DL$ is {\it modus ponens} MP.  
We write $\vdash P$ if $P$ is provable in $\DL$ from the axioms and MP. 

Denial Logic cannot justify its own Denial Axiom.

\begin{prop}\label{PropUnjustified}
$\DL + \left\lbrace\JUST s {(\JUST t P \IMP \NOT P)}\right\rbrace$  is inconsistent.
\end{prop}
\begin{proof}
Suppose that for some justification term $s$, $\PROVES  \JUST s {(\JUST t P \IMP \NOT P)}$.
From the Denial Axiom, 
$\PROVES\JUST s {(\JUST t P \IMP \NOT P)}\IMP \NOT \left(\JUST t P \IMP \NOT P\right).$
By MP, 
$\PROVES\NOT \left(\JUST t P \IMP \NOT P\right)$.
\end{proof} 

\subsection{Negative constant specifications}\label{SubNCS}
A {\it constant specification} CS for a justification logic $L$ is a set of formulas
$e_n:\cdots:e_1:A$ in which $e_1,\ldots,e_1$ are justification constants and $A$ is an axiom of $L$, and which is downward closed: if $e_n:\cdots:e_1:A$ is in CS, then so is $e_{n-1}:\cdots:e_1:A$ \cite{Art2008}. In particular, if $e:A$ is in CS, then so is $A$. An argument similar to that of Proposition \ref{PropUnjustified} shows that nonempty constant specifications for $\DL$ are inconsistent unless they contain no formula of the form $\JUST e P$.  

However, $\DL$ can have what we call negative constant specifications. A 
{\it negative constant specification} NCS for $\DL$ is a collection $\mathcal{C}$ 
of formulas  of the form $\JUST {e_1} {\JUST \cdots {\JUST {e_n} P}}$ or the negation 
of such a formula, where the $e_i$ are justification constants, such that $\DL+NCS$ is consistent
and where $\mathcal{C}$ is closed under the following rules.
\begin{itemize}
	\item[]{Rule 1.} If $\JUST e P\in \mathcal{C}$, then $\NOT P\in\mathcal{C}$.
	\item[]{Rule 2.} If $\NOT \JUST e P\in \mathcal{C}$, then $P\in\mathcal{C}$.
\end{itemize}
 
Rule 1 is motivated by the Denial axiom and Rule 2 is consistent with it. 
If $\DL+\lbrace\NOT R\rbrace$ is consistent, 
then $\lbrace \JUST {e_1} R,\,\NOT R\rbrace$ is  a negative constant specification.

The choice of a language $L$ for $\DL$ together with a negative constant specification (which may be empty) determines 
we call a {\it denial logic}.

\section{Natural semantics of $\DL$}
 
We recall Artemov's natural semantics for $\JL$ \cite{Art2012}.  Let $L$ be a justification logic.  Let $\mathbf{2}=\lbrace0,1\rbrace$ denote the set of truth values, let $\mathbf{Var}$ denote the set of propositional variables of $L$,  let $\mathbf{Tm}$ denote the set of justification terms of $L$, and let $\mathbf{Fm}$ denote the set of $L$-formulas. 
For $X,Y\subseteq\mathbf{Fm}$ we define the set $X\cdot Y$ 
(of consequents of implications in $X$ and antecedents in $Y$) by
\[
X\cdot Y=\lbrace Q | \left(P\IMP Q\right)\in X \AND P\in Y\rbrace.
\]
Also, we define
\[
X\WITH Y = \lbrace P\AND Q | P\in X \AND Q\in Y\rbrace. 
\]

Following Artemov, a {\it modular model} of $\JL$ is a pair of maps, both denoted by $*$, 
of types $\mathbf{Var}\rightarrow \mathbf{2}$ and $\mathbf{Tm}\rightarrow \mathbf{2}^\mathbf{Fm}$ 
respectively, which satisfy the following relationsn \cite{Art2012}.
\begin{eqnarray}\label{J1}
s^*\cdot t^*\subseteq \left[s\cdot t\right]^*\\
\label{J2}
s^*\cup t^*\subseteq \left[s+t\right]^*
\end{eqnarray}
We define $\left(\JUST t P\right)^*=1$ iff $P\in t^*$.  The model $*$ is extended homomorphically 
to Boolean connectives in the obvious way.
We write $\vDash P$ for $P^*=1$. 

The  relation \eqref{J1} corresponds to the Application axiom \eqref{itm:AxApp} and 
the relation \eqref{J2} corresponds to the Sum Axiom \eqref{itm:AxSum} of $\JL$ 
\cite{Art2012}.

\begin{prop}
Let $*$ be a modular model of $\JL$ and let $L$ be a denial logic.  The Evidence Pairing axiom  of $L$ is satisfied by  $*$ if the relation \eqref{J3} holds.
\begin{eqnarray}\label{J3}
s^*\WITH t^*\subseteq \left[s\WITH t\right]^*
\end{eqnarray}
\end{prop}
\begin{proof}
Suppose that $*$ is
a modular model.  If $\MODELS \JUST s P$ and $\MODELS \JUST t Q$ hold, then
$P\in s^*$ and $Q\in t^*$, which implies that $P\AND Q\in s^*\WITH t^*$. 
By \eqref{J3}, $P\AND Q\in [s\WITH t]^*$,
so that $\MODELS \JUST {[s\WITH t]}{P\AND Q}$. It follows that
$\MODELS \JUST s P \AND \JUST t Q\IMP \JUST {[s\WITH t]}{P\AND Q}$ 
holds if \eqref{J3} holds. 
\end{proof}

The image $t^*$ of a justification term is a set of formulas called a {\it justification set}.
Models of $\DL$ satisfy an additional property:  justification sets of modular models
of $\DL$   contain only false formulas.

\begin{prop}\label{PropDenial}
Let  $*$ be a modular model of $\DL$. For every justification term $t$, 
\begin{equation}\label{J4}
t^*\subseteq *^{-1}(0).
\end{equation}
\end{prop}
\begin{proof}
$\vDash \JUST t P\IMP \NOT P$ iff ${\JUST t P}^* = 0$ or $P^*=0$
iff $P\notin t^*$ or $P^*=0$ iff $t^*\subseteq *^{-1}(0)$.
\end{proof}
 
A {\it modular model}  of $\DL$ is a modular model $*$ of $\JL$ that 
satisfies $\eqref{J3}$ and \eqref{J4}.  Given a negative constant 
specification $NCS$  $\DL$,  a model $*$ {\it respects $NCS$} if  all 
formulas of the $NCS$ hold in $*$. This is a translation to $\DL$ of 
the analogous notion for constant specifications defined in  \cite{Art2008,Art2012}.
 
\begin{example}
There are obvious models of $\DL$:  any map $*:\mathbf{Var}\rightarrow\mathbf{2}$ 
can be extended to $\mathbf{Tm}\rightarrow\mathbf{2^{Fm}}$ by setting  $t^*=\emptyset$ 
for any  justification term $t$; from there $*$ is extended
to $\mathbf{Fm}$ in the obvious way. We call such a modular model a
{\it trivial} model. The Denial axiom is  satisfied in all trivial models.  
\end{example}

\subsection{Soundness and completeness of $\DL$}

\begin{theorem}\label{ThmSC}
$\DL+NCS\PROVES \varphi$ if and only if $\vDash\varphi$ holds in every model of $\DL$ 
respecting $NCS$.
\end{theorem}
\begin{proof}
The soundness  $\DL$ follows from Proposition \ref{PropDenial} and Theorem 1 of \cite{Art2012}. 
The proof of completeness is the same as in  \cite{Art2012}.
\end{proof}

\subsection{Inductive construction of models}\label{Con3}
We give an inductive construction of all modular models of $\DL$,
depending on a valuation of propositional variables and a Boolean valued-functional.

This construction of a nontrivial modular model for {\bf DL} will 
extend a map $*:\mathbf{Var}\rightarrow\mathbf{2}$  simultaneously to $\mathbf{Tm}\rightarrow\mathbf{2^{Fm}}$ and to
$*:\mathbf{Fm}\rightarrow\mathbf{2}$. These models will be parametrized by a Boolean-valued functional of type
 $F:\mathbf{2}^\mathbf{Var}\times\tau^3\rightarrow\mathbf{2}$, where the ordinal $\tau$ will be defined.

Since  $\mathbf{Tm}$ is inductively defined, there exist well orderings
on $\mathbf{Tm}$ so that if $\gamma_0,\ldots,\gamma_n$ are ordinals, $t_{\gamma_i}\in\mathbf{Tm}$ and $t_{\gamma_n}=F(t_{\gamma_0},\ldots,t_{\gamma_{n-1}})$ then $\gamma_0,\ldots,\gamma_{n-1}<\gamma_n$. Fix such a well ordering
on $\mathbf{Tm}$. Likewise, $\mathbf{Fm}$ is constructed inductively from $\mathbf{Var}$ and $\mathbf{Tm}$, so a 
well ordering can be defined on $\mathbf{Fm}$ so that for ordinals $\gamma_0,\ldots,\gamma_{n-1},\beta_0,\ldots\beta_m$, if
$P_{\beta_j}\in\mathbf{Fm}$ and $P_{\beta_m}=G(t_{\gamma_0},\ldots,t_{\gamma_{n-1}};P_{\beta_0},\ldots,P_{\beta_{m-1}})$,
then  for $0\le i<n$ and $0\le j<m$, $\max(\gamma_i,\beta_j)<\beta_m$. There may exist ordinals 
$\beta$ for which $P_\beta$
is undefined; we  assume that a given case below holds only when 
relations occurring in them are defined and hold--this applies to the $t_\gamma\in\mathbf{Tm}$ as well.

The well ordering of $\mathbf{Fm}$ implies that 
$\mathbf{Fm}=\bigcup_{\alpha<\tau}\mathbf{Fm}_\alpha$, 
where $\mathbf{Fm}_\alpha=\lbrace P_\beta:\beta<\alpha\rbrace$,
and  where  $\tau$ is the order type of the  chosen well ordering 
of $\mathbf{Fm}$.   For  $\gamma<\alpha<\tau$ and $t_\gamma\in\mathbf{Tm}$, we will inductively 
define the set $t_{\gamma,\alpha}^*\subseteq\mathbf{Fm}_\alpha$ 
and the satisfaction relation $*_\alpha:\mathbf{Fm}_\alpha\rightarrow\mathbf{2}$.

At the end of the construction we set $t_\gamma^* = \cup_{\alpha<\tau} t_{\gamma,\alpha}^*$ 
and likewise extend the $*_\alpha$ to $\mathbf{Fm}$. 
We proceed by induction on $\alpha<\tau$.

\smallskip
\noindent Suppose $\alpha= 0$. Set all $t_{\gamma,0}^*=\emptyset$ and 
let $*_0$ be the empty function $\mathbf{Fm}_0\rightarrow\mathbf{2}$.

\smallskip\noindent
Suppose $\alpha>0$. The inductive hypothesis  is that $*_\sigma$ is defined on $\mathbf{Fm}_\sigma$ 
for $\sigma<\alpha$ and that $t_{\gamma,\sigma}^*$ is defined for 
$\gamma,\sigma<\alpha$.

If $\alpha$ is a limit ordinal, then set
$t_\gamma^* = \cup_{\beta<\alpha} t_{\gamma,\beta}^*$ and let $*_\alpha$ be the union of the 
$*_\beta$ for $\beta<\alpha$. Otherwise, $\alpha$ is a successor ordinal and 
the remaining cases apply. For Boolean connectives 
we show $\IMP$ only as  this is needed to verify the Application axiom; 
the cases of conjunction $\AND$, disjunction $\OR$ and negation $\NOT$ follow the same 
pattern and are not shown. 

Let $\beta,\gamma<\alpha$.

\smallskip\noindent
Case 1. $P_\beta$ is in $\mathbf{Var}$. 

1a) Define
\[
t^{*,\beta}_{\gamma,\alpha}=
\begin{cases}
\lbrace P_\beta\rbrace,&\mathrm{if}\;P_\beta^*=0\AND F(*,\alpha,\beta,\gamma)=1;\\
\emptyset, &\mathrm{otherwise}.
\end{cases}
\]

1b) Set ${P_\beta}^{*_\alpha}={P_\beta}^*$.

\smallskip\noindent 
Case 2.  $P_\beta=P_\mu\IMP P_\nu$ for $\mu,\nu < \alpha$.  The well ordering on
$\mathbf{Fm}$ ensures that $\mu,\nu <\beta< \alpha$. By the inductive hypothesis, 
the partial satisfaction relation $*_\beta$ is defined on 
$\mathbf{Fm}_{\max(\mu,\nu)}$ and hence on $P_\mu$ and $P_\nu$.

2a)  Define
\[
t^{*,\beta}_{\gamma,\alpha}=
\begin{cases}
\lbrace P_\beta\rbrace,&\mathrm{if}\;P_\mu^{*_\beta}=1\AND P_\nu^{*_\beta}=0
 \AND F(*,\alpha,\beta,\gamma)=1;\\
\emptyset, &\mathrm{otherwise}.
\end{cases}
\]
         
2b) Define
\[
P_\beta^{*_\alpha}=
\begin{cases}
0,&\mathrm{if}\;P_\mu^{*_\beta}=1\AND P_\nu^{*_\beta}=0;\\
1,&\mathrm{otherwise}.
\end{cases}
\]

\smallskip\noindent 
Case 3. $P_\beta=\JUST {t_\gamma} {P_\delta}$ for $\gamma, \delta<\alpha$. 
Again, $\gamma, \delta<\beta<\alpha$ and
by the inductive hypothesis, $P_\delta^{*_\beta}$ is defined.

3a) Define
\[
t^{*,\beta}_{\gamma,\alpha}=
\begin{cases}
\lbrace P_\delta\rbrace,&\mathrm{if}\;P_\delta^{*_\beta}=0
 \AND F(*,\alpha,\beta,\gamma)=1;\\
\emptyset, &\mathrm{otherwise}.
\end{cases}
\]

3b)  Define
\[
P_\beta^{*_\alpha}=
\begin{cases}
1,&\mathrm{if}\; P_\delta^{*_\beta}=0\AND F(*,\alpha,\beta,\gamma)=1;\\
0,&\mathrm{otherwise}.
\end{cases}
\]

\smallskip\noindent 
Case 4. None of the above ($P_\beta$ is undefined). Do nothing.

\smallskip\noindent 
Prologue. At the end of stage $\alpha$, we define for $\gamma<\alpha$,
\[
t^*_{\gamma,\alpha} =
\bigcup_{\beta<\alpha}t^{*,\beta}_{\gamma.\alpha}\cup\bigcup_{\sigma<\alpha}t^*_{\gamma,\sigma}
\]
Before proceeding to stage $\alpha+1$, we may need to update the justification set $t_{\gamma,\alpha}^*$ to 
ensure that the Sum axiom \eqref{itm:AxSum} and the Evidence Pairing
axiom \eqref{itm:AxPairing} continue to hold. For each $\gamma<\alpha$ there are three cases.

\smallskip
Case a) There exist $\mu,\nu<\gamma$ such that 
$t_{\gamma}= t_{\mu}+t_{\nu}$. If $ t_{\mu,\alpha}^*\cup t_{\nu,\alpha}^*\subseteq t_{\gamma,\alpha}^*$
do nothing; otherwise redefine
\[
t_{\gamma,\alpha}^*:= t_{\mu,\alpha}^*\cup t_{\nu,\alpha}^*\cup 
\bigcup_{\beta<\alpha}t^{*,\beta}_{\gamma.\alpha}\cup\bigcup_{\sigma<\alpha}t^*_{\gamma,\sigma}.
\]

\smallskip
Case b) There exist $\mu,\nu<\gamma$ such that $t_{\gamma}= t_{\mu}\WITH t_{\nu}$.
If $t_{\mu,\alpha}^*\WITH t_{\nu,\alpha}^*\subseteq  t_{\gamma,\alpha}^*$ do nothing; otherwise
redefine
\[
t_{\gamma,\alpha}^*:= \left(t_{\mu,\alpha}^*\WITH t_{\nu,\alpha}^*\right)\cup 
\bigcup_{\beta<\alpha}t^{*,\beta}_{\gamma.\alpha}\cup\bigcup_{\sigma<\alpha}t^*_{\gamma,\sigma}.
\]

\smallskip
Case c) None of the above. This includes the possibility that $t_\gamma$ is undefined. Do nothing.

\smallskip\noindent 
(End of construction \ref{Con3}).

\begin{prop}
For each Boolean-valued functional 
$F:\mathbf{2}^\mathbf{Var}\times\tau^3\rightarrow\mathbf{2}$
and for each valuation $*:\mathbf{Var}\rightarrow\mathbf{2}$,
the construction \ref{Con3} yields a  modular model 
of {\bf DL}.
\end{prop}  
\begin{proof}
We show that conditions \eqref{J1},\,\eqref{J2} and \eqref{J3} are satisfied.  Note that at each stage $\alpha$ only formulas that evaluate to $0$ (false) are enumerated into the image of any justification term. Hence \eqref{J3} is satisfied, and by Proposition \eqref{PropDenial}, the Denial axiom \eqref{itm:AxDenial} holds.

For \eqref{J1}, note that for any $s,t\in\mathbf{Tm}$, $s^*\cdot t^*=\emptyset$. Otherwise there exist $P,Q\in\mathbf{Fm}$ such that
$Q\in s^*\cdot t^*$, 
$P\IMP Q\in s^*$ and $P\in t^*$. But then $(P\IMP Q)^*=0$, which forces $P^*=1$ and $Q^*=0$. But since $P\in t^*$, by construction,
$P^*=0$. This is a contradiction. Therefore \eqref{J1} holds vacuously, 
and the Application axiom \eqref{itm:AxApp} holds.

The Prologue at stage $\alpha$ of Construction \eqref{Con3} ensures that the 
Sum Axiom \eqref{itm:AxSum} and the Evidence Pairing Axiom \eqref{itm:AxPairing} hold.
\end{proof}  

\subsection{Examples of modular models}

\begin{example}
Trivial modular models are obtained by setting $F(*,\alpha,\beta,\gamma)=0$ for all valuations $*:\mathbf{Var}\rightarrow\mathbf{2}$ and for all ordinals $\alpha,\beta,\gamma<\tau$.
\end{example}

\begin{example}
 Setting $F=1$ yields modular models such that $t^*=*^{-1}(0)$ for every justification term $t$. These models are maximal.
\end{example}
 
\begin{example}
This example will illustrate that the inclusion in \eqref{J2} can be strict \cite{Art2012}.  
Define the functional
$F:\mathbf{2}^\mathbf{Var}\times\tau^3\rightarrow\mathbf{2}$ by
\[
F(*,\alpha,\beta,\gamma)=
\begin{cases}
1,&\mathrm{if}\; t_\gamma\;\mathrm{contains}\;\mathrm{a}\;+;\\
0,&\mathrm{otherwise}.
\end{cases}
\]

Let $*$ be a valuation and suppose that the formula $P$ is false in the model determined 
by $F$ and $*$. Then $\nVdash [x+y]:P \rightarrow x:P\OR y:P$.

This syntactical criterion ensures that for any false formula $P$, if the justification
$t$ has sufficient complexity, then $\JUST t P$ is true.
\end{example}

\begin{example}  
Every modular model of {\bf DL} arises through the choice of a pair $(F, v)$ 
consisting of a functional $F$ and a valuation 
$v:\mathbf{Var}\rightarrow\mathbf{2}$.  Given a negative constant specification 
there is a choice of a functional $F$ that realizes it.  For $\alpha$ large enough and for
$P_\beta=\JUST {e_1} {\JUST \cdots {\JUST {e_n} P}}$ in the specification, where $e_1=t_\gamma$, 
set $F(*,\alpha,\beta,\gamma)=1$. If $P_\beta=\NOT\JUST {e_1} {\JUST \cdots {\JUST {e_n} P}}$,
then set  $F(*,\alpha,\beta,\gamma)=0$.
\end{example}

\begin{example}
Let $\DL{\degree}$ be denial logic without the Pairing axiom \ref{itm:AxDenial} and the binary operation $\WITH$ on justification terms. The
completeness theorem holds for this logic. Also, the construction above goes through without mention of the operation $\WITH$ or
the Pairing axiom to yield an inductive construction of models of $\DL\degree$. 
Let $NCS=\lbrace \JUST a A,\NOT A,\JUST b B,\NOT B\rbrace$ and
assume that $\DL\degree+NCS$ is consistent. We may use the construction and the procedure of the previous example
to build a model of  $\DL\degree+NCS$ such that for each justification term $t$, $t^*\subseteq \lbrace A, B\rbrace$.  It follows that there is
no justification term $s$ such that $\MODELS \JUST s {A\AND B}$ (otherwise $A\AND B\in s^*$, contrary to the construction). 
Hence  there is no justification term $s$ such that $\PROVES_{\DL\degree+NCS} \JUST s {A\AND B}$. 
This shows that the Evidence Pairing axiom is independent of the other axioms of $\DL$. 
\end{example}

\section{Philosophical Interpretation of $\DL$}  
Artemov has characterized Justification Logic $\JL$ with an empty constant 
specification as the ``logic of general (not necessarily factive) justifications 
for an absolutely skeptical agent for whom no formula is provably justified'' \cite{Art2008}.   
Denial Logic $\DL$ models an agent whose  justified beliefs are false, who cannot avow his own
propositional attitudes, who is capable  of believing logical contradictions and for whom 
even tautologies of classical logic cannot be justified. The soundness and completeness 
of $\DL$ and the number of its nontrivial modular models suggest that  $\DL$ is suitable 
as a logic of philosophical skepticism that goes beyond  the skeptical challenges of 
Descartes' {\it Meditations} and Putnam's brains in vats \cite{descartes1991,HP1981}. 

In his {\it First Meditation}, Descartes writes that he can still reason about 
arithmetic and geometry even if he is dreaming: ``For whether I am awake or 
asleep, two and three added together are five, and a square has no more than 
four sides. It seems impossible that such transparent truths should incur 
any suspicion of being false'' \cite{descartes1991}. In our applications of $\DL$ below,  our 
agents will stop short of believing contradictions and disavowing their own beliefs.

In {\it Reason, Truth and History}, Hilary Putnam presented 
a modern formulation of the skeptical challenge posed by the 
evil genius of Descartes' {\it Meditations} \cite{HP1981}. 
Putnam's argument that we could not be brains in a vat 
provoked a vigorous philosophical response, e.g. 
\cite{HP1981,JM1984,PS1984,JH1985,ALB1986,MD1991,%
ALB1992,DC1993,FS1993,YS1994,ALB1995,GF1995,HWN1998,SS1999,BW2000,HWN2000,%
BCJ2003,AB2006}. For Putnam, the thesis that we are, 
have always been, and will always be brains in a vat is self-refuting on 
semantic grounds. The argument remains controversial.

\subsection{Brains in vats under $\DL$}\label{SecBiVDL} 
We interpret Putnam's thought experiment so that 
the logic of belief of our brain in a vat concerning its sensory 
experience is denial logic $\DL$.  In this interpretation, the world 
consists of a bio-computing facility  supplying electrical impulses 
to a community of brains in vats with nutrients and appropriate 
cabling to the computing facility.  The programming of the facility 
determines the sense experiences of the envatted brains as its 
simulation runs \cite{HP1981}.  We assume that logic of belief of an envatted brain  
of its sense experience is given by a choice of language for $\DL$ 
together with the choice of a negative constant specfication $CS$ 
of the justified false beliefs the brain holds about its experience.

More explicitly, we assume that for any sense experience the 
computing facility induces within an envatted brain, there are one or more 
formulas $E$ of $\DL$ asserting something that the brain believes about that 
experience (e.g., ``I am not a brain in a vat''). 
Also, we assume that for every such formula $E$ that the envatted brain believes for 
some reason indicated by its experience,  there is a justified formula $\JUST s E$ of $CS$ 
with justification term $s$, such that in every model $*$ of $\DL+CS$, the 
interpretation $s^*$ of  $s$ in the model $*$ contains the formulas representing 
the account the brain would provide of its experience, were the brain to attempt to 
justify its belief (e.g., ``I am walking outside'', ``the air is cool'', 
``I hear the cellphone ringing'' and so on).

Following Putnam's thought experiment, we will assume that the sentence
``I am not a brain in a vat'' is represented in $\DL+CS$ 
as $\JUST s E$ for some justification term $s$. 
In $\DL+CS$, $\NOT E$ holds; that is, the brain is indeed a brain in a vat. 
In general we assume that $\DL+CS$ can formally represent the epistemic state of 
an envatted brain as the computing facility determines the brain's sense experience,
at least to to one order of belief. Beliefs about beliefs will be addressed in the sequel.

Our intention is that $\DL+CS$ formalizes the epistemic situation of the 
envatted brain, as expressed by Ludwig in this passage \cite[p.35--36]{KW1992}.
\begin{quotation}
[...] the brain in the vat is thinking `There's a tree' and means what we 
do by that, and fails to be, in one sense, thinking about a tree simply 
because his assumption that he is in perceptual contact with a tree is false. 
The brain in the vat, then, far from having mostly true beliefs in virtue of 
its not being in causal contact with trees and tables and chairs and the like, 
and so not thinking about such things, has mostly false beliefs precisely
because he fails, when talking about things about him, to be thinking about 
trees and tables and chairs and so on. 
\end{quotation}

Given $\DL+CS$, we use the knowledge extraction operator $\mathrm{OK}$ of
Artemov and Kuznets to  derive a new logical system with an interpretation 
$*$ such that if $\{\JUST s E,\, \NOT E\}\subseteq CS$ 
(e.g., where $\JUST s E$ is ``I am not a brain in a vat for reason $s$''), 
then $E$ will be true in $*$.  
If $L$ is a justification logic and if $CS$ is a  (negative) 
constant specification, we define the knowledge extraction operator $\mathrm{OK}$ as follows
\cite{SARK2009}.
\[
\OK{CS}:=
\lbrace E : \exists t\in\mathbf{Tm},\, \PROVES_{L+CS} \JUST t E\rbrace.
\]

A negative constant specification $CS$ for $\DL$ is {\it coherent} if  
for every justification term $t$ such that $\PROVES_{\DL+CS} \JUST t E$, 
there exists a model $*$ of $\JL$ such that 
$\MODELS E$. By $\JL$ we mean justification logic with the same language 
as $\DL+CS$, but without the specification. 

\begin{prop}\label{PropFinite}
Suppose that $CS$ is a coherent constant specification for $\DL$.   
Then every finite subcollection of $\OK{CS}$ is satisfiable in some  
modular model of $\JL$. 
\end{prop}
\begin{proof}
Suppose otherwise. Then there is a finite sequence $\JUST{t_1}{E_1},\ldots,\JUST{t_n}{E_n}$ of justified formulas with  $E_i\in\OK{CS}$ such that  
$\PROVES_{\DL+CS} \JUST {t_i} {E_i}$ for $1\le i\le n$, and such that 
$((\ldots(E_1\AND E_2)\AND\cdots)\AND E_n)$ is unsatisfiable
in $\JL$.    By logic, 
\[
\PROVES_{\DL+CS}((\ldots(\JUST {t_1}{E_1}\AND\JUST {t_2}{E_2})\AND\cdots)\AND\JUST{t_n}{E_n})
\]

It follows from the Evidence Pairing axiom\ref{itm:AxPairing} that  there exists a 
justification term $t = [[\ldots[t_1\WITH t_2]\WITH\cdots]\WITH t_n] $ such that 
\[
\PROVES_{\DL+CS}\JUST t {((\ldots(E_1\AND E_2)\AND\cdots)\AND E_n)}.
\] 
Since $CS$ is coherent, there is a model $*$ of $\JL$  in which  $\MODELS ((\ldots(E_1\AND E_2)\AND\cdots)\AND E_n)$. This is a contradiction.
\end{proof}

\begin{theorem}[Blue Pill Theorem]\label{CorCompact}
$\JL+\OK{CS}$ has a model $*$.
\end{theorem}
\begin{proof}
By Proposition \ref{PropFinite} and the compactness of $\JL$. Compactness follows from the
completeness theorem for $\JL$ \cite{Art2012}. 
\end{proof}

Under the assumption that the collection of justified beliefs of an 
envatted brain is coherent, there exists a model $*$ in 
which ``I am not a brain in a vat'' comes  out true, as well as  
all of the other justified false assertions of the negative 
constant specification $CS$.

The $\JL$ model $*$ was derived from Proposition \ref{PropFinite}  under the
assumption that $\DL+CS$ is consistent (which implies that it has a model) 
and that $CS$ is  coherent. The mathematical argument makes no reference
to Putnam's causal theory of reference \cite{HP1981}. It does not offer much assurance that
we are not brains in a vat, however.

Proposition \ref{PropFinite} and  Corollary \ref{CorCompact} hold in more general justification logics that satisfy an evidence pairing property.  In general we have the following.

\begin{corollary}\label{PropFinite2}
Suppose that $CS$ is a coherent constant specification for a justification logic $L$.   
Suppose further that whenever $\PROVES_{L+CS}\JUST r A \AND \JUST s B$ there exists a justification term
$t$ such that $\PROVES_{L+CS} \JUST t {A\AND B}$.
Then every finite subcollection of $\OK{CS}$ is satisfiable in some  
modular model of $\JL$ with the same language as $L$. 
\end{corollary}

\begin{corollary}\label{CorCompact2}
$\JL+\OK{CS}$ has a model $*$.
\end{corollary}

\section{Multi-modal denial logic}\label{SecMulti}
Crispin Wright's discussion of  propositional attitudes in {\it On Putnam's Proof that we are not Brains-in-a-vat} suggests why $\DL$ is not suitable for modeling beliefs about belief \cite[p. 79]{CW1994}.

\begin{quotation}
It is part of the way we ordinarily think about self-consciousness  
that we  regard  the  contents  of  a subject's  contemporary 
propositional attitudes as something which, to use the standard term  
of  art,  they  can  {\it avow}--something  about which 
their judgements are credited with a strong, though defeasible authority  
which does not rest on reasons or evidence. 
\end{quotation}

If knowledge is true justified belief, $\DL$ cannot model agents that 
know anything.  Nevertheless, $\DL$ is one starting point for logical models 
of skeptical thought experiments. To handle agents who can believe theorems 
of classical logic and verify them, and who can avow their own propositional attitudes, we can 
extend  $\DL$  to $\DLP$, the algebraic fibring  of Denial Logic $\DL$ with the  
Logic of Proofs $\LoP$ constrained by justification logic $\JL$.  This is a pushout construction,
first described in \cite{Sernadas1999}.\footnote{More generally, we suggest that the Humean  distinction between ``matters of fact'' and ``relations of ideas'' may be modeled by the algebraic fibring of two justification logics: one appropriate for matters of fact and one appropriate to relations of ideas (e.g., analytic propositions) 
\cite{hume1912enquiry}.}  
\[
\begin{array}{ccc}
\mathrm{JL} & \rightarrow & \mathrm{LP} \\ 
\downarrow & & \downarrow \\ 
\mathrm{DL} & \rightarrow & \mathrm{DL}\oplus_\mathrm{JL}\mathrm{LP}
\end{array}
\]
We indicate the construction in our case, which amounts to the fusion 
of modal propositional Hilbert calculi.

\subsection*{Syntax and semantics of $\DLP$}

The syntax of $\DLP$ amends that of $\DL$ as follows. Justification constants and variables are signed and have the form $c_i^\sigma, x_i^\sigma$ where $\sigma\in\lbrace +,-\rbrace$. The signed justification constant $c_i^+$ (variable $x_i^+$) is   {\it positive}, and the signed justification constant $c_i^-$ (variable $x_i^-$) is   {\it negative}. A {\it positive (negative) justification term} of $\DLP$ is an expression of the form
\[
t^\sigma \DCOLON= c_i^\sigma\, |\, x_j^\sigma\,
|\, \left[t^\sigma+t^\sigma\right]\,
|\, \left[t^\sigma\cdot t^\sigma\right]\, 
|\, \left[t^-\WITH t^-\right]\, 
|\, \BANG t^+
\]
where $\sigma\in\lbrace+,-\rbrace$, $t^+$ is positive, and $t^-$ is negative.  The set $\mathbf{Tm}$ of justification terms is the disjoint union of the set $\mathbf{Tm}^+$ positive justification terms and the set  $\mathbf{Tm}^-$ of negative justification terms.

The definition of a formula of  $\DLP$ amends the definition of a formula of denial logic $\DL$ as follows.
A {\it formula} $A$ of $\DL^\pm$  is an expression of the form
\[
A \DCOLON= \FALSE\,  |\, A_i \, |\, \left(A\AND A\right) \,| \left(A\OR A\right) \,|\, \left(A\IMP A\right) \,|\,  \JUST {t^+} A
\,|\,  \JUST {t^-} A
\]
where $t^+$ is a positive justification term, $t^-$ is a negative justification term, and where $A_i$ is a propositional variable.
The axioms of $\DLP$ are the same as $\DL$ and $\LoP$ with the proviso that  the Denial and Evidence Pairing axioms are restricted to negatively justified formulas, factivity and introspection are restricted to positively justified formulas, and the signs of the terms appearing in the Sum and Application axioms must be the same. Equivalently, the following axioms are those of $\DLP$.
\begin{enumerate}
	\item Application. $\JUST {s^\sigma} {\left(P\IMP Q\right)} \IMP (\JUST {t^\sigma} P \IMP \JUST {[s^\sigma\cdot t^\sigma]} Q)$ 
		\label{itm:AxSgnApp}
	\item Sum.  $\JUST {s^\sigma} P \IMP \JUST {[s^\sigma+t^\sigma]} P,\quad
                        \JUST {t^\sigma} P \IMP \JUST {[s^\sigma + t^\sigma]} P$i
			\label{itm:AxSgnSum}
		\item Denial. $\JUST {s^-} P\IMP \NOT P$\label{itm:AxSgnDenial}
		\item Pairing. $\left(\JUST {s^-} P\AND\JUST {t^-} Q\right) 
\IMP \JUST {[s^-\WITH t^-]}{\left(P\AND Q\right)}$
   \label{itm:AxSgnPairing}
\item Positive Factivity. $\JUST{t^+}P\IMP P$\label{itm:AxPosFactivity}
\item Positive Introspection. $\JUST{t^+}P\IMP\JUST {\BANG t^+} {\JUST {t^+} P}$
	\label{itm:AxPosIntro}
\end{enumerate}
where $\sigma\in\lbrace+,-\rbrace$, $s^-$ is negative, and $t^+$ is positive. 
The rule of inference in $\DLP$ is modus ponens.

\begin{prop}\label{PropDisjoint}
$\DLP\PROVES \NOT \left( \JUST {s^-} P\AND\JUST {t^+} P \right)$, 
where $s^-$ is negative, $t^+$ is positive and where $P$ is a formula.
\end{prop}
\begin{proof}
Immediate from Axioms \ref{itm:AxSgnDenial} and \ref{itm:AxPosFactivity}. 
\end{proof}

Constant specifications for $\DLP$ generalize those of $\JL$ and $\DL$. A {\it constant specification} 
for $\DLP$ is a collection $\mathcal{C}$ of formulas of the form $\JUST {e_1} {\JUST \cdots {\JUST {e_n} P}}$ and negations of such formulas,
where the $e_i$ are justification constants, such that $\left(\DLP\right)+\mathcal{C}$ is consistent 
and where $\mathcal{C}$ is closed under the following rules.

\begin{itemize}
	\item[]{Rule 1.} If $\JUST t P\in \mathcal{C}$ and $t$ is positive, then $P\in\mathcal{C}$.
	\item[]{Rule 2.} If $\JUST s P\in \mathcal{C}$ and $s$ is negative, then $\NOT P\in\mathcal{C}$.
	\item[]{Rule 3.} If $\NOT \JUST t P\in \mathcal{C}$ and $t$ is positive, then $\NOT P\in\mathcal{C}$.
	\item[]{Rule 4.} If $\NOT \JUST s P\in \mathcal{C}$ and $s$ is negative, then $P\in\mathcal{C}$.
\end{itemize}

This definition is motivated  as follows. A set $\mathcal{C}$ consisting of 
formulas only  of the form $\JUST {e_1} {\JUST \cdots {\JUST {e_n} P}}$ with 
all $e_i$ positive is just a constant specification in the usual sense. Rule 1 
is forced by  the Positive Factivity axiom \eqref{itm:AxPosFactivity}, and Rule 3 is consistent with with it. Dually, Rule 2 is forced by signed
denial, and Rule 4 is consistent with it.
  
A {\it modular model}  of $\DLP$ is  a pair of maps, both denoted by $*$, 
of types $\mathbf{Var}\rightarrow \mathbf{2}$ and $\mathbf{Tm}^-\cup\mathbf{Tm}^+=\mathbf{Tm}\rightarrow \mathbf{2}^\mathbf{Fm}$   that satisfy the following conditions.  
\begin{eqnarray}\label{DLP1}
\left(s^\sigma\right)^*\cdot \left(t^\sigma\right)^*\subseteq \left[s^\sigma\cdot t^\sigma\right]^*\\
\label{DLP2}
\left(s^\sigma\right)^*\cup \left(t^\sigma\right)^*\subseteq \left[s^\sigma+t^\sigma\right]^*\\
\label{DLP2a}
\left(s^-\right)^*\WITH \left(t^-\right)^*\subseteq \left[s^-\WITH t^-\right]^*\\
\label{DLP3}
\left(t^-\right)^*\subseteq \left(*|\mathbf{Tm}^-\right)^{-1}(0)\\
\label{DLP4}
\left(t^+\right)^*\subseteq \left(*|\mathbf{Tm}^+\right)^{-1}(1)\\
P\in\left(t^+\right)^*\IMP\JUST{t^+} P\in \left(\BANG t^+\right)^*
\end{eqnarray}
where $s^\sigma,t^\sigma\in\mathbf{Tm}^\sigma$, $\sigma\in\lbrace+,-\rbrace$.

If $CS$ is a constant specification for $\DLP$, then  a model $*$ {\it respects $CS$} if  all formulas of the $CS$ hold in $*$ \cite{Art2008,Art2012}. The analog of  Theorem \ref{ThmSC}, the soundness and completeness theorem,  
holds for $\DLP$.

\begin{theorem} $\DLP$ is complete.\label{ThmDLP}
\end{theorem}

\begin{proof}
Suppose that $\varphi$ is a formula of $\DLP$ such that 
\begin{eqnarray}\label{Unprovable}
	\nvdash_{\DLP} \varphi.
\end{eqnarray}
Let $L$ be a justification logic, and let $\mathbf{Fm}(L)$ denote the set of formulas
of $L$.  Let $V$ be a  set of new propositional variables $X_{\JUST {s^-}\psi}$ for 
each negatively justified formula $\JUST {s^-}\psi$ of $\DLP$. Let $L'$ denote the 
justification logic obtained from the $\LoP$ fragment of $\DLP$ by adjoining the set $V$ of 
new variables and closing under Boolean connectivies and positive justification terms. 
If $\JUST s \psi\in \mathbf{Fm}(L')$, then $s$ is a positive justification term. 

Define by induction a transformation $T:\mathbf{Fm}(\DLP)\rightarrow\mathbf{Fm}(L')$ as follows.
\begin{itemize}
	\item[] $T(E) = E$ if $E$ is a propositional variable;
	\item[] $T(\NOT E) = \NOT T(E)$;
	\item[] $T(D\circ E) = T(D)\circ T(E)$ if $\circ\in\{\AND,\OR,\IMP\}$;		
	\item[] $T(\JUST {s^+} E) = \JUST {s^+} {T(E)}$ for $s^+$ positive;
	\item[] $T(\JUST {s^-} E) = X_{\JUST {s^-}{T(E)}}$ for $s^-$ negative.
\end{itemize}
The logic $L'$ together with the $T$-images of the axioms of $\DLP$
and with Modus Ponens as the rule of inference is a justification logic satisfying
the axioms of the Logic of Proofs. This is because the $T$-images
of classical logic axioms in $\DLP$ are precisely the substitution instances of
those axioms in $\mathbf{Fm}(L')$; and  similarly the $T$-images of
the application, factivity, sum and proof checker axioms of the $\LoP$ fragment
of $\DLP$ are precisely the substitution instances of those axioms in $\mathbf{Fm}(L')$.
Hence for $\psi\in\mathbf{Fm}(\DLP)$, 
$\vdash_{\DLP}\psi \Leftrightarrow \vdash_{L'}T(\psi)$.  
It follows from (\ref{Unprovable}) that $\nvdash T(\varphi)$. 
By the completeness theorem for $LoP$, there is a modular model $*'$ of $L'$
such that $*'\nvDash T(\varphi)$ \cite{KuznetsStuder}.

Define a model $*$ of $\DLP$ from $*'$ as follows. 
If $\psi$ is a propositional variable of $\DLP$, take $*\psi = *'\psi$.
If ${t}^{+}$ is a positive justification term, take 
${(t^{+})}^*={(t^{+})}^{*'}$.
If $s^-$ is a negative justification term, define ${(s^-)}^*$  to be the
set such that for each $\psi\in\mathbf{Fm}(\DLP)$, 
\[
	\psi\in{(s^-)}^*\Leftrightarrow *'\vDash E_{\JUST{s^-} {T(\psi)}}.
\]

By induction, if $\psi$ is a formula of $\DLP$, then
\begin{eqnarray}\label{Transfer}
*\vDash \psi \Leftrightarrow *'\vDash T(\psi).
\end{eqnarray}
For example, $*\vDash \JUST {t^-}\psi \Leftrightarrow 
\psi\in {(s^-)}^* \Leftrightarrow
*'\vDash T(\JUST {t^-} \psi)$.

Moreover, $*$ satisfies conditions (\ref{DLP1}) through (\ref{DLP4}) of subsection \ref{DLP1}.
For example, the $T$-image of the Denial axiom ensures that $*$ satisfies condition (4) above.
If $*\vDash\JUST{s^-}\psi$ holds, then so does 
$*'\vDash E_{\JUST{s^-} {T(\psi)}}$. Since the $T$-image of the Denial axiom
holds, $*'\vDash E_{\JUST{s^-} {T(\psi)}}\IMP \NOT T(\psi)$, and hence
$*'\vDash\NOT T(\psi)$ holds. But by (\ref{Transfer}), $*\vDash \NOT\psi$, which
yields (4). The other cases are similar. Hence $*$ is a model of $\DLP$ in which
$*\nvDash\varphi$.
\end{proof}

\subsection{Internalization in $\DLP$.} 
A constant specification $CS$ for a justification logic $\JL$  
is {\it axiomatically appropriate} if 
for each axiom $A$ of $\JL$, there is a justification constant
$e_1$ such that $\JUST {e_1} A$ is in $CS$, and if $CS$ is downward
closed (cf. subsection \ref{SubNCS}) \cite{Art2008}.
The Logic of Proofs $\LoP$ can internalize its deductions. More generally,
we have the following.
\begin{theorem}[Theorem 1, \cite{Art2008}]\label{ThmInternal}
For each axiomatically appropriate constant specification $CS$,
$\JL+CS$ enjoys the Internalization Property:

\centerline{If $\PROVES F$, then
$\PROVES \JUST p F$ for some
justification term $p$.}
\end{theorem}
Recall that $\DL$ with a nontrivial constant specification is inconsistent. {\it A fortiori},
$\DL$ with an axiomatically appropriate constant specification is inconsistent. 
This fact motivated our definition of negative constant specifications 
for $\DL$ and our generalization of negative constant specifications 
to $\DLP$. 
Accordingly, the definition of axiomatically appropriate constant 
specifications in $\DLP$ must be restricted to constant specifications 
$CS$ with positive justification constants applied to the axioms of $\DLP$. 
With this restriction,  $\DLP+CS$  satisfies the internalization property.

We also note that in justification logics with axiomatically appropriate constant 
specifications, the Evidence Pairing axiom is unnecessary. In such a system, there 
is a justification constant $a$ such that
\begin{equation}\label{EqJust}
\JUST a{\left(P\IMP \left(Q\IMP(P\AND Q)\right) \right)}
\end{equation}
holds. If $\JUST s P$ holds, then by the Application axiom and modus ponens we have
$\JUST {[a\cdot s]}{\left(B\IMP(P\AND Q)\right)}$. And if $\JUST t Q$ holds, application
and modus ponens gives us $\JUST{ [[a\cdot s]\cdot t]} (P\AND Q)$. However, \eqref{EqJust} is
inconsistent with $\DL$.

\section{Anthropogenic Global Warming Denial in $\DLP$}
Earth scientists often speak of the indicators produced by 
their biogeophysical  models \cite{VFML2000,VDGR2005,Sav2000}. 
The scientific modeling, remote sensing and economic literature is 
replete with this usage: ecologists speak of indicators of 
ecosystem stress;  meteorologists speak of atmospheric  indicators of climate 
change, economists speak of economic indicators and so on 
\cite{MalUvegesTurk2002,Reichler_2009,TangTank2003}. 
The scientific vocabulary of indicators of biogeophysical models suggests 
reading   $\JUST t E$ in $\JL$ as   ``$t$ indicates $E$.'' This reading is more
natural for our applications than ``term $t$ is justification for $E$.''

The logic $\DLP$ can model an agent who believes that climate 
models are as ``close to falsification [...] as mathematically 
possible,'' but who allows that $\mathrm{CO}_2$ is a greenhouse 
gas \cite{curry2012}. For this agent, this the only 
scientific statement pertaining to global warming that can 
be justified.  We think of the $\DL$ fragment of $\DLP$ as 
the logic of indicators of some climate model. The first 
assertion is expressed by the Denial axiom, which holds in 
the $\DL$ fragment of $\DLP$.  For example, if $E$ is some  
statement about the environment, such as ``global warming 
is accelerating,'' then $t:E$ is the statement that $t$ is 
indication of some climate model that $E$.  The second 
statement is positively justified in the $\LoP$ fragment of 
$\DLP$ as $s:C$, where $s$ is a positive justification term. 

\begin{prop}
In $\DLP$, if $s:C$ with $s$ positive, and if $t:E$ with $t$ 
negative, then there is no positive justification term $j$ 
such that $\JUST j {\left(C\IMP E\right)}$. 
\end{prop}
\begin{proof}
Suppose there exists such a $j$. The Positive Application axiom 
asserts that for $j, s$ positive
justification terms, $[j\cdot s]$ is also positive, and
\begin{eqnarray}
\JUST j {\left(C\IMP E\right)} \IMP\left( \JUST s C \IMP \JUST {[j\cdot s]} E\right).
\end{eqnarray}
By MP twice, $\JUST {[j\cdot s]} E$. This together with 
the assumption $\JUST t E$ contradicts Proposition \ref{PropDisjoint}.
\end{proof}

The agent asserts that  the only thing scientists can justify 
is that $CO_2$ is a greenhouse gas. This is modeled  by 
adjoining to $\DLP$ the formulas $\JUST c C$ and $C$, 
where  $c$ is a positive justification constant. In effect, 
the agent accepts  $c:C$ and $C$ as axioms.
However, ``climate models are as close to falsification
[...] as is mathematically possible,'' which in $\DLP$ is 
represented as the Denial axiom in the $\DL$ fragment. In 
this context, in which the $\DL$ fragment of $\DLP$ models
the logic of indicators of some climate model, the Denial axiom asserts
that every indicator produced by that climate model is wrong. 
This is enough to rule out any scientific conclusion that 
might follow from the concession that $CO_2$ is a 
greenhouse gas.


One can say more about the epistemic state of the agent in
a larger system, $\DLJP$, in which the  middle summand
can represent justified beliefs that may be true or false. In this case
the agent believes that the denial axiom holds for climate models.
This is modeled 
in $\DL\oplus_{\JL}\oplus_{\JL}\LoP$ 
by adjoining $\JUST {e} {\left(\JUST {s^-} E\IMP \NOT E\right)}$
where $e$ is a ``neutral'' justification constant from the middle $\JL$ summand.

Also, the agent believes that he knows that the Denial axiom holds for climate models.
This is expressed by the formula 
$\JUST {e} {\JUST {t^+} {\left(\JUST {s^-} E\IMP \NOT E\right)}}$, 
where $t^+$ is a positive justification constant.

\section{Envatted brains under $\DLP$}
We augment the situation and logical apparatus of subsection \ref{SecBiVDL}  
by stipulating that the logic of justified belief of the computer generated 
sensory experiences of the envatted brains is definable within the $\DL$ 
fragment of the logic $\DLP$. Further, this logic can model  brains  that 
have positively justified beliefs  about their negatively justified beliefs
about their sense experience. An envatted brain  may correctly believe that 
it believes that its sensory experience indicates that it is an embodied human 
(and not a brain confined to a vat). In $\DLP$ this epistemic state is expressed 
as a formula of the form $\JUST {t^+}{\JUST {s^-} E}$, in which 
$t^+$ represents the positively justified belief that the negatively
justified belief represented by $s^-$ indicates $E$.  
From $\JUST {t^+}{\JUST {s^-} E}$ Positive Factivity yields ${\JUST {s^-} E}$, and 
from this Denial yields $\NOT E$. We do not stipulate that the Denial axiom is 
positively justified.     
The analogs of Proposition \ref{PropFinite} and the Blue Pill Theorem \ref{CorCompact}
hold in $\DLP$.

\section*{Acknowledgements} We thank Sergei Artemov for helpful comments.

\printbibliography

\end{document}